\UseRawInputEncoding
\documentclass[final,12pt]{elsarticle}
\usepackage{ragged2e}
\justifying
\usepackage{bm}
\usepackage{geometry}
\usepackage{amsmath}
\usepackage{mathrsfs}
\usepackage{amsthm}
\usepackage{amssymb}
\newtheorem{theorem}{Theorem}[section]

\newtheorem{lemma}{Lemma}[section]

\newtheorem{problem}{Problem}[section]

\numberwithin{equation}{section}

\newenvironment{proof of 1.7}{\noindent{\emph{Proof of Theorem 1.7 (1).}}\ }{\hfill $\square$\par}
\newenvironment{proof of 1.10}{\noindent{\emph{Proof of Theorem 1.10.}}\ }{\hfill $\square$\par}

\usepackage{indentfirst}
\begin{document}
	\begin{frontmatter}  
		\title{Spectral extrema of graphs: Forbidden star-path forests\,\tnoteref{titlenote}}  
		\tnotetext[titlenote]{This work was supported by the National Nature Science Foundation  of China (Nos. 11871040, 11971180, 12271337).}    
		\author{Yanni Zhai$^1$}  
		\author{Xiying Yuan$^{1,}$\corref{correspondingauthor}}
		\author{Lihua You$^2$}  
		\cortext[correspondingauthor]{Corresponding author. \\
			Email address: yannizhai2022@163.com (Yanni Zhai), xiyingyuan@shu.edu.cn (Xiying Yuan), \\ylhua@scnu.edu.cn (Lihua You).}   
		\address{$^1$Department of Mathematics, Shanghai University, Shanghai 200444, P.R. China} 
		\address{$^2$ School of Mathematical Sciences, South China Normal University, Guangzhou, 510631, P.R. China} 
		\begin{abstract}  
		A path of order $n$ is denoted by $P_n$, and a star of order $n$ is denoted by $S_{n-1}$. A star-path forest is a forest whose connected components are paths and stars.  In this paper we determine the maximum spectral radius of  graphs that contain no copy of $kS_{\ell-1}\cup P_{\ell}$, $k_1S_{2\ell -1}\cup k_2P_{2\ell}$ or $kS_{4}\cup 2P_{5}$ for $n$ appropriately large. 
		\end{abstract}   
	\end{frontmatter}
	\section{Introduction}
	In this paper, we consider undirected graphs without loops or multiedges. The order of a graph $G=\left(V(G),\,E(G)\right)$ is the number of its vertices, and the size of a graph $G$ is the number of its edges, denoted by $e(G)$. The adjacent matrix $A(G)=(a_{ij})$ of $G$ is a matrix, where $a_{ij}=1$ if $v_i$ is adjacent to $v_j$, and 0 otherwise. The spectral radius of $G$ is the largest eigenvalue of $A(G)$, denoted by $\rho (G)$. 
	For a vertex $v\in V(G)$, the neighborhood of $v$ in $G$ is denoted by $N_{G}(v)=\{u\in V(G):  uv\in E(G)\}$ or simply $N(v)$ and $N[v]=N(v)\cup \{v\}$.  Denote $N_G^2(v)$ or simply $N^2(v)$ by the set of vertices at distance two from $v$ in $G$.

	Given two vertex-disjoint graphs $G$ and $H$, the union of graphs $G$ and $H$ is the graph $G\cup H$ with vertex set $V(G)\cup V(H)$ and edge set $E(G)\cup E(H)$. In particular, $G=kH$ is the vertex-disjoint union of $k$ copies of $H$. The join of $G$ and $H$, denoted by $G\vee H$, is the graph obtained from $G\cup H$  by adding all edges between $V(G)$ and $V(H)$. For a graph $G$ and its subgraph $H$, $G-H$ denotes the subgraph induced by $V(G)\setminus V(H)$. For $U\subseteq V(G)$, let $G[U]$ be the subgraph of $G$ induced by $U$.
	A path of order $n$ is denoted by $P_{n}$, and a star of order $n$ is denoted by $S_{n-1}$. 
	Let $S_{n,\,h}$ be the graph of order $n$ obtained by joining a clique of order $h$ with an independent set of order $n-h$, i.e., $S_{n,\,h}=K_h\vee \overline{K}_{n-h}$. Let $S^{+}_{n,\,h}$ be the graph obtained from $S_{n,\,h}$ by adding one edge, i.e., $S^{+}_{n,\,h}=K_h\vee \left(K_2\cup \overline{K}_{n-h-2}\right)$ .

	A graph is $H$-free if it does not contain a copy of $H$ as a subgraph. The Tur\'{a}n number ex$(n,\,H)$ is the maximum number of edges in a graph of order $n$ which is $H$-free. We denote by Ex$(n,\,H)$ the set of $H$-free graphs of order $n$ with ex$(n,\,H)$ edges, and call a graph in Ex$(n,\,H)$ an extremal graph for $H$. 
	
	In 1956, Erd\H{o}s and Gallai \cite{1959} studied Tur\'{a}n numbers of paths. Then in 2008, Balister and Gy\H{o}ri \cite{2008-P} gave an improvement of the extremal graph theorem for $P_{k}$. In 2019, Lan et al. \cite{2019-L} determined the Tur\'{a}n number of a star.
	A linear forest is a forest whose connected components are paths. A star forest is a forest whose connected components are stars. A star-path forest is a forest whose connected components are paths and stars. In 2011, Bushaw and Kettle \cite{2011} determined the Tur\'{a}n numbers of $kP_{\ell}$ for sufficiently large $n$, which was extended by Lidicik\'{y}, Liu and Palmer \cite{2013}. Yuan and Zhang \cite{2017,2021} determined the Tur\'{a}n numbers of linear forests containing at most one odd path for all $n$. In 2022, Li, Yin and Li \cite{2022-L} determined the Tur\'{a}n numbers of  star forests. Very recently, Fang and Yuan \cite{2022-F} determined the Tur\'{a}n numbers of $kS_{\ell -1}\cup P_{\ell}$, $k_1S_{2\ell-1}\cup k_2P_{2\ell }$ and $kS_{4}\cup 2P_{5}$ for $n$ appropriately large.
	
	\begin{theorem}[{Balister and Gy\H{o}ri~\cite[]{2008-P}}]\label{lu-jie}
		Let $G$ be a connected graph of order $n$ containing no path of order $k+1$, $n>k\geq 3$. Then 
		\begin{equation*}
			e(G)\leq {\rm{max}}\;\bigg\{\binom{k-1}{2}+(n-k+1),\;\binom{\lceil(k+1)/2\rceil}{2} +\bigg\lfloor\frac{k-1}{2}\bigg\rfloor \left(n-\bigg\lceil\frac{k+1}{2}\bigg\rceil\right)\bigg\}.
		\end{equation*}	
	\end{theorem}
	
	\begin{theorem}[{Lan, Li, Shi and Tu~\cite[]{2019-L}}]\label{xing-jie1}
		If $\ell\geq 3$ and $n\geq \ell +1$, then 
		\begin{equation*}
			{\rm{ex}}(n,\,S_{\ell})\leq \bigg{\lfloor}\frac{(\ell -1)n}{2}\bigg{\rfloor},
		\end{equation*}
	with one extremal graph is the $(\ell -1)$-regular graph of order  $n$.
	\end{theorem}
	
	 \begin{theorem}[{Li, Yin and Li~\cite[]{2022-L}}]\label{xing-jie2}
	 	If $k\geq 2$ and $\ell \geq 3$, then 
	 \small{	\begin{equation*}
	 		{\rm{ex}}(n,\,kS_{\ell})=\left\{
	 		\begin{aligned}
	 			&\binom{n}{2},       \quad\quad\quad\quad \quad \quad\;  \quad \quad \quad \quad \quad \quad \quad\quad\quad\quad\quad \quad \quad \quad \;\;\;if\; n<k(\ell +1),\\
	 			&\binom{k\ell +k-1}{2}+\binom{n-k\ell -k+1}{2}, \quad \quad  \;  \quad\quad  \quad\quad\quad\quad\; if \;k(\ell +1)\leq n\leq (k+1)\ell +k-1,\\
	 			&\binom{k\ell +k-1}{2}+\bigg{\lfloor}\frac{(\ell -1)(n-k\ell -k+1)}{2}\bigg{\rfloor},  \quad \quad \quad \quad \quad if\;(k+1)\ell +k\leq n<\frac{k\ell ^2+2k\ell +2k-2}{2},\\
	 			&\binom{k-1}{2}+(n-k+1)(k-1)+\bigg{\lfloor} \frac{(\ell -1)(n-k+1)}{2} \bigg{\rfloor}, \;if\;n\geq \frac{k\ell ^2+2k\ell+2k-2 }{2}.
	 		\end{aligned}\right.
	 	\end{equation*}}
	 \end{theorem}

	 \begin{theorem}[{Fang and Yuan~\cite[]{2022-F}}]\label{F-1}
	 	Suppose $n=k+d(\ell -1)+r\geq (\ell ^2-\ell +1)k+\frac{\ell ^2+3\ell -2}{2}$, where $k\geq 1$, $\ell \geq 4$ and $0\leq r< \ell -1$. Then
	 	\begin{equation*}
	 		{\rm{ex}}\left(n,\,kS_{\ell -1}\cup P_{\ell}\right)=\left(k+\frac{\ell}{2}-1\right)n-\frac{k^2+(\ell -1)(k+r)-r^2}{2}.
	 	\end{equation*}
 	Moreover,
	 	\begin{equation*}
	 		{\rm{Ex}}(n,\,kS_{\ell -1}\cup P_{\ell})=\left\{
	 		\begin{aligned}
	 			&\{K_k\vee (dK_{\ell -1}\cup K_r),\;S_{n,\;k+\frac{\ell }{2 }-1}\}, \;\; \ell \;is\; even,\; and\;r=\frac{\ell }{2}\;or\;r=\frac{\ell -2}{2},\\
	 		&\{K_k\vee (dK_{\ell -1}\cup K_r)\}, \;\;\;\;\;\;\;\;\;\;\;\;\;\;\;\;\,\; otherwise.
	 		\end{aligned}\right.
	 	\end{equation*}
	 \end{theorem}
 Based on the result of Theorem \ref{F-1}, we have 
 \begin{equation}\label{upbound}
 		{\rm{ex}}(n,\,kS_{\ell -1}\cup P_{\ell })\leq \left(k+\bigg{\lfloor}\frac{\ell}{2}\bigg{\rfloor}-\frac{1}{2}\right)(n-1).
 \end{equation}

 \begin{theorem}[{Fang and Yuan~\cite[]{2022-F}}]\label{F-3}
 		Suppose $n\geq (4\ell ^2-2\ell +1)k_1+(2\ell ^2+3\ell -4)k_2+3$, where $k_1\geq 1$, $k_2\geq 2$ and $\ell \geq 2$. Then
 	\begin{equation*}
 		{\rm{Ex}}(n,\,k_1S_{2\ell -1}\cup k_2P_{2\ell})=\{S_{n,\,k_1+\ell k_2-1}\}.
 	\end{equation*}
 \end{theorem}

\begin{theorem}[{Fang and Yuan~\cite[]{2022-F}}]\label{F-4}
	Suppose $n\geq 21k+38$ where $k\geq 1$. Then 
		\begin{equation*}
		{\rm{Ex}}(n,\,kS_{4}\cup 2P_5)=\{S^{+}_{n,\,k+3}\}.
	\end{equation*}
\end{theorem}
In 2010, Nikiforov \cite{2010} proposed  the spectral counterpart of Tur\'{a}n type extremal problem.
\begin{problem}\label{p1}
	Given a graph $H$, what is the maximum $\rho (G)$ of a graph $G$ of order $n$ without $H$ as a subgraph?
\end{problem}

This problem has been  intensively investigated in the literature for many classes of graphs. Guiduli \cite{1998} and  Nikiforov \cite{2017-N} independently studied the case $H=K_r$. In 2010, Nikiforov \cite{2010} studied the case $H$ is a path or cycle of specified length.  Chen, Liu and Zhang  in 2019 \cite{2019} studied  the case $H$ is a linear forest, and in 2021 they \cite{2021-M} studied the case $H$ is a star forest. For other classes of graphs, the readers may be referred to \cite{2020,2019-G,2023,2012, arxiv2}. Motivated by Problem \ref{p1}, we will give the maximum  value of the spectral radius and characterize corresponding extremal graphs for some kinds of star-path forests.
 Denote by Ex$_{sp }(n,\,H)$  the set of $H$-free graphs of order $n$ with maximum spectral radius.
 
 For any graph $F$, if its extremal graph is $S_{n,\,p}$ (resp. $S^{+}_{n,\,p}$), then we will prove that for appropriate $n$, its spectral extremal graph is also $S_{n,\,p}$ (resp. $S^{+}_{n,\,p}$).
 
\begin{theorem}\label{1}
	\begin{enumerate}[(1)]
		\item For any graph $F$, suppose $n\geq m$ and {\rm{Ex}}$(n,\,F)=\{S_{n,\,p}\}$. Then when $n\geq\, ${\rm{max}}$\,\{2^{4p},\,m^2\}$, we have {\rm{Ex}}$_{sp }(n,\,F)=\{S_{n,\,p}\}$.
		\item For any graph $F$, suppose $n\geq m$ and   {\rm{Ex}}$(n,\,F)=\{S^{+}_{n,\,p}\}$. Then when $n\geq{\rm{max}}\,\{2^{4p},\,m^2\}$, we have 
		{\rm{EX}}$_{sp }(n,\,F)=\{S^{+}_{n,\,p}\}$.
	\end{enumerate}
\end{theorem}

 The spectral extremal graphs for $k_1S_{2\ell -1}\cup k_2P_{2\ell}$ and $kS_{4}\cup 2P_{5}$ can be determined  based on the results of Theorems \ref{F-3} \textendash\,\ref{1}.

\begin{lemma}
	Suppose $k_1\geq 1$, $k_2\geq 2$, $\ell \geq 2$ and $n\geq$ {\rm{max}} $\{2^{4(k_1+\ell k_2-1)},\,[(4\ell ^2-2\ell +1)k_1+(2\ell ^2+3\ell -4)k_2+3]^2\}$. Then
	\begin{equation*}
		{\rm{Ex}}_{sp}(n,\,k_1S_{2\ell -1}\cup k_2P_{2\ell})=\{S_{n,\,k_1+\ell k_2-1}\}.
	\end{equation*}
\end{lemma}

\begin{lemma}
	Suppose $k\geq 1$ and $n\geq $ {\rm{max}} $\{2^{4(k+3)},\,(21k+38)^2\}$. Then 
	\begin{equation*}
		{\rm{Ex}}_{sp}(n,\,kS_{4}\cup 2P_{5})=\{S^{+}_{n,\,k+3}\}.
	\end{equation*}
\end{lemma}

Now we will consider spectral extremal graphs for $kS_{\ell-1}\cup P_{\ell}$. If $\ell =2$, then $kS_{\ell-1}\cup P_{\ell}=(k+1)P_2$ holds. In 2007, Feng, Yu and Zhang \cite{2007} gave the spectral extremal graph for $kP_{2}$.
\begin{theorem}[{Feng, Yu and Zhang~\cite[]{2007}}]
	Suppose $k\geq 1$ and $n\geq 2k$.
	\begin{enumerate}[(1)]
		\item If $n=2k$ or $2k+1$, then ${\rm{Ex}}_{sp}(n,\,kP_{2})=\{K_n\}$ holds.
		\item If $2k+2 \leq n< 3k+2$, then ${\rm{Ex}}_{sp}(n,\,kP_{2})=\{K_{2k+1}\cup \overline{K_{n-2k-1}}\}$ holds.
		\item If $n=3k+2$, then ${\rm{Ex}}_{sp}(n,\,kP_{2})=\{K_{2k+1}\cup \overline{K_{n-2k-1}},\,K_{k}\vee \overline{K_{n-k}}\}$ holds.
		\item If $n>3k+2$, then ${\rm{Ex}}_{sp}(n,\,kP_{2})=\{K_k\vee \overline{K_{n-k}}\}$ holds.
	\end{enumerate}
\end{theorem}

If $\ell =3$, then $kS_{\ell-1}\cup P_{\ell}=(k+1)P_3$ holds. 
In 2019, Chen, Liu and Zhang \cite{2019} gave the spectral extremal graph for $kP_3$. Set $F_{n,\,k}=K_{k-1}\vee (dK_2\cup K_s)$, where $n-(k-1)=2d+s$ and $0\leq s<2$.
\begin{theorem}[{Chen, Liu and Zhang~\cite[]{2019}}]
	Suppose $k\geq 2$ and $n\geq 8k^2-3k$. Then 
	\begin{equation*}
		{\rm{Ex}}_{sp}(n,\,kP_{3})= \{F_{n,\,k}\}.
	\end{equation*}
\end{theorem}

For $\ell\geq 4$, we give the spectral extremal graphs for $kS_{\ell -1}\cup P_{\ell}$.
	\begin{theorem}\label{2}
		Suppose $k\geq 1$, $\ell \geq 4$ and $n\geq 8\left(k+\lfloor\frac{\ell}{2}\rfloor\right)^3t^8$ where  $t=(\ell ^2-\ell +1)k+\frac{\ell ^2+3\ell -2}{2}$. 
		\begin{enumerate}[(1)]
			\item If $\ell $ is even, then ${\rm{Ex}}_{sp}(n,\,kS_{\ell -1}\cup P_{\ell})=S_{n,\,\frac{2k+\ell-2 }{2}}$ holds.
			\item If $\ell $ is odd, then ${\rm{Ex}}_{sp}(n,\,kS_{\ell -1}\cup P_{\ell})=S^{+}_{n,\,\frac{2k+\ell-3 }{2}}$ holds.
		\end{enumerate}
	\end{theorem}


	\section{Proof of Theorem \ref{1}}
We would like to point out that the proof of Theorem \ref{1} is inspired by \cite{2019}, and  the following two lemmas are very useful to prove Theorem \ref{1}.
\begin{theorem}[{Hong, Shu and Fang~\cite[]{2001},  Nikiforov~\cite[]{2002}}]\label{rho upper}
	Let $G$ be a graph of order $n$ with the minimum degree $\delta=\delta (G)$ and $e=e(G)$. Then
	\begin{equation*}
		\rho(G)\leq \frac{\delta -1+\sqrt{8e-4\delta n+(\delta +1)^2}}{2}.
	\end{equation*}
\end{theorem}

\begin{theorem}[{Nikiforov~\cite[]{2010}}]\label{subgraph}
Suppose $c\geq 0$, $p\geq 2$, $n\geq 2^{4p}$. If $\delta (G)<p$ and 
	\begin{equation*}
		\rho (G)\geq \frac{p-1+\sqrt{4pn-4p^2+c}}{2},
	\end{equation*}
	then there exists a subgraph $H$ of order $q\geq \lfloor \sqrt{n}\rfloor $ satisfying one of the following conditions:
	\begin{enumerate}[(1)]
		\item $q=\lfloor \sqrt{n}\rfloor $ and $\rho (H)>\sqrt{(2p+1)q}$;
		\item $q>\lfloor \sqrt{n}\rfloor$, $\delta(H)\geq p$ and $\rho (H)>\frac{p-1+\sqrt{4pq-4p^2+c+2}}{2}$.
	\end{enumerate}
\end{theorem}

\begin{proof of 1.7}
In \cite{2010}, it is pointed out that 
	\begin{equation}\label{S-jie}
		\rho (S_{n,\,p})= \frac{p-1+\sqrt{4pn-4p^2+(p-1)^2}}{2}.
	\end{equation}
	Suppose $G\in {\rm{Ex}}_{sp}(n,\,F)$. Since $S_{n,\,p}$ is $F$-free, we have 
	\begin{equation}\label{1-1}
		\rho (G)\geq \rho (S_{n,\,p})=\frac{p-1+\sqrt{4pn-4p^2+(p-1)^2}}{2}.
	\end{equation}

	{\textbf{Case 1.}} $\delta (G)\geq p$. 
	
	It is easy to see the function $f(x)=\frac{x-1+\sqrt{8e-4xn+(x+1)^2}}{2}$ is decreasing with respect to $x$. Then by Theorem \ref{rho upper} we have
	\begin{equation*}
		\rho (G)\leq \frac{\delta (G)-1+\sqrt{8e(G)-4\delta (G)n+(\delta (G)+1)^2}}{2}\leq \frac{p-1+\sqrt{8e(G)-4pn+(p+1)^2}}{2}.
	\end{equation*}
	Together with (\ref{1-1}), we have $e(G)\geq pn-\frac{p^2+p}{2}=e(S_{n,\,p})$. Then $G=S_{n,\,p}$ holds.
	
	{\textbf{Case 2.}}	$\delta (G)< p$. 
	
	If we take $c=(p-1)^2$ in Theorem \ref{subgraph}, then there is a graph $H\subseteq G$ of order $q\geq \lfloor \sqrt{n}\rfloor$ satisfying one of the following conditions:
	\begin{enumerate}[(1)]
		\item $q=\lfloor \sqrt{n}\rfloor$ and $\rho (H)>\sqrt{(2p+1)q}$;
		\item $q>\lfloor \sqrt{n}\rfloor$, $\delta(H)\geq p$ and $\rho (H)>\frac{p-1+\sqrt{4pq-4p^2+c+2}}{2}>\frac{p-1+\sqrt{4pq-4p^2+(p-1)^2}}{2}$.
	\end{enumerate}

	When $q=\lfloor \sqrt{n}\rfloor$ and $\rho (H)>\sqrt{(2p+1)q}$, then
	\begin{equation*}
		2e(H)=tr\left(A^2(H)\right)\geq \rho ^2(H)>(2p+1)q>2e(S_{q,\,p}).
	\end{equation*}
Noting that $q=\lfloor \sqrt{n}\rfloor \geq m$ and {\rm{Ex}}$(q,\,F)=\{S_{q,\,p}\}$ hold, then $H$ contains a copy of $F$ and then $G$ contains a copy of $F$. This is a contradiction.

	When $q>\lfloor \sqrt{n}\rfloor$, $\delta(H)\geq p$ and $\rho (H)>\frac{p-1+\sqrt{4pq-4p^2+(p-1)^2}}{2}$, applying Theorem \ref{rho upper}, we have
	\begin{equation*}
		\rho (H)\leq \frac{p-1+\sqrt{8e(H)-4pq+(p+1)^2}}{2}.
	\end{equation*}
Combining with $\rho (H)>\frac{p-1+\sqrt{4pq-4p^2+(p-1)^2}}{2}$, we have $e(H)>qp-\frac{p^2+p}{2}=e(S_{q,\,p})$. Hence $H$ contains a copy of $F$ and then $G$ contains a copy of $F$. That is a contradiction.

The proof of Theorem \ref{1}  (2) is similar to that of (1).
\end{proof of 1.7}

	\section{Some auxiliary results}
	 To prove Theorem \ref{2}, we need a rough bound on the size of $kS_{\ell -1}\cup P_{\ell}$-free bipartite graphs.
	\begin{lemma}\label{u1}
		Let $\ell \geq 4$ and $n$ be positive integers. If $G$ is a $P_{\ell}$-free bipartite graph of order $n$, then 
		\begin{equation*}
			e(G)\leq \left(\bigg{\lfloor}\frac{\ell }{2} \bigg{\rfloor}-1\right)n.
		\end{equation*}
	\end{lemma}
	\begin{proof}
		When $n<\ell$, we have $e(G)\leq \lfloor \frac{n}{2}\rfloor\lceil\frac{n}{2}\rceil \leq \left(\lfloor\frac{\ell}{2}\rfloor-1\right)n.$ When $n\geq \ell$, we show the inequality by using induction on $n$. 
		When $n=\ell$, since $G$ is bipartite and $P_{\ell}$-free, we have $$e(G)\leq \bigg{\lfloor}\frac{n}{2}\bigg{\rfloor}\bigg{\lceil}\frac{n}{2}\bigg{\rceil} -1\leq \left(\bigg{\lfloor}\frac{\ell}{2}\bigg{\rfloor}-1\right)n.$$ Suppose that $n>\ell $ and the inequality holds for all $\ell \leq n'<n$. If $G$ is connected, then by Theorem \ref{lu-jie}, we have 
		$$e(G)\leq {\rm{max}}\left\{\binom{\ell -2}{2}+(n-\ell +2),\,\binom{\lceil\frac{\ell }{2}\rceil}{2}+\bigg{\lfloor}\frac{\ell -2}{2}\bigg{\rfloor} \left(n-\bigg{\lceil}\frac{\ell}{2}\bigg{\rceil}\right) \right\}\leq \left(\bigg{\lfloor}\frac{\ell}{2}\bigg{\rfloor}-1\right)n. $$
	 Now we suppose $G_1,\,G_2,\,\cdots,\,G_s$ are the connected components of $G$ with $\lvert V(G_i)\rvert =n_i$. If $n_i\geq \ell$, then by induction hypothesis  $e(G_i)\leq \left(\lfloor \frac{\ell}{2}\rfloor -1\right)n_i$ holds.
	 If $n_i<\ell$, then we have $e(G_i)\leq \lfloor \frac{n_i}{2}\rfloor \lceil\frac{n_i}{2}\rceil\leq \left(\lfloor\frac{\ell}{2} \rfloor-1 \right)n_i$.
Therefore, we obtain $$e(G)= e(G_1\cup \cdots \cup G_s)\leq \sum_{i=1}^{s}\left(\bigg\lfloor\frac{\ell}{2} \bigg\rfloor-1 \right)n_i\leq \left(\bigg\lfloor\frac{\ell}{2} \bigg\rfloor-1\right)n.$$ 
\end{proof}
	\begin{lemma}\label{u2}
		Let $G$ be a $kS_{\ell -1}\cup P_{\ell}$-free bipartite graph of order $n\geq \left(k+\lfloor \frac{\ell}{2}\rfloor -1\right)^2-(k+\lfloor \frac{\ell}{2}\rfloor -2)\ell+\ell ^2k+\ell ^2-\ell $ with $k\geq 0$ and $\ell \geq 4$. Then
		\begin{equation*}
			e(G)\leq \left(k+\bigg{\lfloor} \frac{\ell}{2}\bigg{\rfloor} -1\right)n.
		\end{equation*}
	\end{lemma}
	\begin{proof}
		We show the inequality by using induction on $k$. When $k=0$, the result holds by Lemma \ref{u1}. Suppose that $k\geq 1$ and the conclusion holds for all $k'<k$. Write $h=k+\lfloor\frac{\ell}{2}\rfloor -1$. Let $H$ be a  $kS_{\ell -1}\cup P_{\ell}$-free bipartite graph of order $n$ with maximum size. Since the complete bipartite graph $K_{h,\,n-h}$ is $kS_{\ell -1}\cup P_{\ell}$-free, we have 
		\begin{equation*}
			\begin{split}
		e(H)&\geq h(n-h)> \left(k+\frac{\ell }{2}-2\right)n-\frac{(k-1)^2}{2}-\frac{(\ell -2)(k-1)}{2}\\
		&\geq \binom{k-1}{2}+(n-k+1)(k-1)+\bigg{\lfloor}\frac{(\ell -2)(n-k+1)}{2}\bigg{\rfloor} \\
		&= {\rm{ex}}(n,\,kS_{\ell -1}).
		\end{split}
		\end{equation*} 
		Hence $H$ contains a copy of $kS_{\ell-1}$. The fact that $H$ is $kS_{\ell -1}\cup P_{\ell}$-free implies $H-S_{\ell-1}$ is $(k-1)S_{\ell -1}\cup P_{\ell}$-free. By induction hypothesis we have $e(H-S_{\ell-1})\leq (h-1)(n-\ell)$.
	Let $m_0$ be the number of edges incident with the vertices of a $S_{\ell -1}$ in $H$. Then we have
		\begin{equation*}
				m_0=e(H)-e(H-S_{\ell -1})\geq h(n-h)-(h-1)(n-\ell)=n-h^2+(h-1)\ell.
		\end{equation*}
		Then  each copy of $S_{\ell -1}$ in $H$ contains a vertex of degree at least $\frac{n-h^2+(h-1)\ell}{\ell}$. Let 
		$U\subseteq V(H)$ with $\lvert U\rvert =k$ and each vertex in $U$ belongs to distinct $S_{\ell -1}$ with degree at least $\frac{n-h^2+(h-1)\ell}{\ell}$.  Set $\overline{U}=V(H)\setminus U$, $W=\cup _{u\in U}N(u)$ and $W_0=W\cap \overline{U}$ . For any $u\in U$, we have 
		\begin{equation*}
			d_{H[W_0]}(u)\geq \frac{n-h^2+(h-1)\ell }{\ell}-(k-1).
		\end{equation*}
	Then $\lvert W_0\rvert \geq \frac{n-h^2+(h-1)\ell }{\ell}-(k-1)$ holds. If $H[\overline{U}]$ contains a copy of $P_{\ell}$, we set $W_1=W_0\setminus V(P_{\ell})$, then we have
		\begin{equation*}
			\lvert W_1\rvert \geq \lvert W_0\rvert-\ell \geq \frac{n-h^2+(h-1)\ell }{\ell}-(k-1)-\ell\geq (\ell -1)k.
		\end{equation*}
	Then we may find $k$ copies of $S_{\ell -1}$ in $H-P_{\ell}$ with $k$ center vertices in $U$ and leaves vertices in $W_1$, then $H$ contains a copy of $kS_{\ell -1}\cup P_{\ell}$, a contradiction. Therefore, $H[\overline{U}]$ is $P_{\ell}$-free and bipartite. By Lemma \ref{u1}, we have $e(H[\overline{U}])\leq \left(\lfloor \frac{\ell}{2}\rfloor-1\right)(n-k)$. Then we deduce 
		\begin{equation*}
			\begin{split}
			e(H)&\leq e\left(H[\overline{U}]\right)+k(n-k)+\frac{k^2}{4}\\
			&\leq \left(\bigg{\lfloor} \frac{\ell}{2}\bigg{\rfloor}-1\right)(n-k)+k(n-k)+\frac{k^2}{4}\\
			&\leq \left(k+\bigg{\lfloor} \frac{\ell}{2}\bigg{\rfloor}-1\right)n.
			\end{split}
		\end{equation*}
			Therefore we have $e(G)\leq e(H)\leq \left(k+\lfloor \frac{\ell}{2}\rfloor-1\right)n$.
	\end{proof}


   	Suppose  $S_1,\,\cdots,\,S_{k}$ are $k$ finite sets, and the following inequality was proved by \cite{2020}.
   	\begin{equation}\label{cup}
   		\lvert S_{1}\cap \cdots \cap S_{k}\rvert \geq \sum_{i=1}^{k}\lvert S_{i}\rvert -(k-1)\lvert \cup^{k} _{i=1}S_i\rvert.
   	\end{equation}
  The following formulas  are introduced in \cite{arxiv2}.
   For a graph $G$ and any two vertex subsets  $A$ and $B$ of $V(G)$, $e(A,\,B)$ denotes the number of the edges of $G$ with one end vertex in $A$ and the other in $B$. 
   \begin{equation}\label{,1}
   	e(A,\,B)=e(A,\,B\setminus A)+e(A,\,A\cap B)=e(A,\,B\setminus A)+2e(G[A\cap B])+e(A\setminus B,\, A\cap B).
   \end{equation}
   \begin{equation}\label{,2}
   	e(A,\,B)\leq e(G[A\cup B])+e(G[A\cap B])\leq 2e(G),
   \end{equation}
   \begin{equation}\label{,3}
   	e(A,\,B)\leq \lvert A\rvert \lvert B\rvert .
   \end{equation}

	\section{Some characterizations of the graphs in EX$_{sp}(n,\,kS_{\ell -1}\cup P_{\ell})$}
 Suppose $G\in{\rm{Ex}}_{sp }(n,\,kS_{\ell -1}\cup P_{\ell})$. 	For $V_1,\,V_2\subseteq V(G)$ and $V_1\cap V_2=\emptyset$,  $G[V_1,\,V_2]$ denotes the induced bipartite graph with one partite set $V_1$ and the other partite set $V_2$.
 In this section we always suppose $k\geq 1$, $\ell \geq 4$,  $h=k+\lfloor\frac{\ell }{2}\rfloor-1$, $t=(\ell ^2-\ell +1)k+\frac{\ell ^2+3\ell -2}{2}$, $\alpha =\frac{1}{2(h+1)t^2}$ and $n\geq \frac{t^2}{\alpha ^3}$. The fact that $S_{n,\,h}$ is $kS_{\ell -1}\cup P_{\ell}$-free implies 
 \begin{equation}\label{3-1}
 	\rho(G) \geq \rho(S_{n,\,h})=\frac{h-1+\sqrt{4hn-4h^2+(h-1)^2}}{2}\geq \sqrt{hn}.
 \end{equation} 
 
 Firstly we will show $G$ is connected. Suppose to the contrary that $G$ is disconnected and $G_1$ is a component of $G$ with 
 $\rho(G_1)=\rho(G)\geq \sqrt{hn}$. Let $u\in V(G_1)$ with the maximum  degree in $G_1$, $G'$ be the graph obtained from $G_1$ by attaching a pendent edge at $u$ and $n-\lvert V(G_1)\rvert -1$ isolated vertices. Then $G'$ is $kS_{\ell -1}\cup P_{\ell}$-free. Otherwise there is a copy of $kS_{\ell-1}\cup P_{\ell}$ in $G_1$ as $d_{G_1}(u)\geq \rho (G_1)\geq \sqrt{hn}\geq \ell-1$. However, the fact $\rho(G')>\rho (G_1)=\rho(G)$ contradicts $G\in{\rm{Ex}}_{sp }(n,\,kS_{\ell -1}\cup P_{\ell})$. Hence $G$ is connected. Let \textbf{x} be the Perron vector of $G$, i.e., $A(G)\textbf{x}=\rho (G)\textbf{x}$ and $\Vert \textbf{x}\Vert _2=1$. Set $x_z=$max$\{x_v:\,v\in V(G)\}$.


	
 Write $\rho =\rho(G)$ for convenience.
Let $R=\{v\in V(G):x_v>\alpha x_z\}$, and $\overline{R}=V(G)\setminus R$. We will evaluate the cardinality of $ R$. 
\begin{lemma}\label{3-2}
		 $\lvert R\rvert \leq 2\sqrt{hn}$.
\end{lemma}
\begin{proof}
	For each vertex $v\in R$, we have
	\begin{align}\label{2.1}
	\sqrt{hn}\alpha x_z&\leq 	\rho x_v=\sum_{u\in N(v)}x_u\leq \sum_{u\in (N(v)\cap R)}x_z+\sum_{u\in (N(v)\setminus R)}\alpha x_z \notag \\ &=d_{G[R]}(v)x_z+d_{G[\overline{R}]}(v)\alpha x_z.
	\end{align}
Summing the inequality (\ref{2.1}) for all $v\in R$, we have
\begin{equation*}
		\begin{split}
			\lvert R\rvert \sqrt{hn}\alpha x_z\leq \sum_{v\in R}d_{G[R]}(v)x_z+\sum_{v\in R}d_{G[\overline{R}]}(v)\alpha x_z
			=2e(G[R])x_z+e(R,\,\overline{R})\alpha x_z.
		\end{split}
\end{equation*}
If $\lvert R\rvert > 2\sqrt{hn}$, we have $\lvert R\rvert > t$ as $n\geq \frac{t^2}{\alpha^3}$. Since $G[R]$  is $kS_{\ell -1}\cup P_{\ell}$-free, by (\ref{upbound}) we have $e(G[R])\leq (h+\frac{1}{2})\lvert R\rvert$. Furthermore, by Lemma \ref{u2} and the fact that the bipartite graph $G[R,\,\overline{R}]$ is $kS_{\ell -1}\cup P_{\ell}$-free, $e(R,\,\overline{R})\leq hn$ holds. Hence  we obtain
\begin{equation}\label{R}
	\lvert R\rvert  \sqrt{hn}\alpha\leq 2\left(h+\frac{1}{2}\right)\lvert R\rvert +\alpha hn.
\end{equation}
Since $n\geq \frac{t^2}{\alpha ^3}\geq \frac{(4h+2)^2}{\alpha ^2}$, we have $ \sqrt{hn}\alpha-2\left(h+\frac{1}{2}\right)\geq \frac{1}{2} \sqrt{hn}\alpha$. Then from (\ref{R}) we have
\begin{equation*}
	\lvert R\rvert\leq \frac{\alpha hn}{ \sqrt{hn}\alpha-2\left(h+\frac{1}{2}\right)}\leq \frac{\alpha hn}{\frac{1}{2} \sqrt{hn}\alpha}=2\sqrt{hn},
\end{equation*}
which is a contradiction. Therefore $\lvert R\rvert\leq 2\sqrt{hn}$ holds.
\end{proof}

Furthermore, we set  $R'=\{v\in V(G):x_v>4\alpha x_z\}$.  

\begin{lemma}\label{3-3}
	For each $v$ in $R'$, $d_G(v)>\frac{1}{3}\alpha n$ holds.
\end{lemma} 
\begin{proof}
	First we prove a claim.

	\textbf{Claim.}
	For any vertex $v$ in $R$,	if $d_G(v)\leq \frac{1}{3}\alpha n$, then $e\left(G[(N(v)\cup R)\setminus\{v\}]\right)\leq \frac{(5h-2)\alpha n}{6}$ holds.
		
	When $\lvert (N(v)\cup R)\setminus\{v\}\rvert <t$, we have
	\begin{equation*}
		  e\left(G[(N(v)\cup R)\setminus\{v\}]\right)\leq \frac{t(t-1)}{2}\leq \frac{(5h-2)\alpha n}{6},
	\end{equation*}
	  where the last inequality holds as $n\geq \frac{t^2}{\alpha^3}$ and $\alpha=\frac{1}{2(h+1)t^2}$.
	
	Now we suppose $\lvert (N(v)\cup R)\setminus\{v\}\rvert \geq t$.  Since $G[N(v)\cup R]$ is $kS_{\ell-1}\cup P_{\ell}$-free, by (\ref{upbound}), we have 
	\begin{equation*}
		e(G[N(v)\cup R])\leq \left(h+\frac{1}{2}\right)\left(d_{G}(v) + \lvert R\rvert\right),
	\end{equation*}
which implies 
\begin{equation*}
	\begin{split}
	e\left(G[(N(v)\cup R)\setminus\{v\}]\right)&\leq \left(h+\frac{1}{2}\right)\left(d_{G}(v) + \lvert R\rvert\right)-d_G(v)\\
	&\leq\frac{1}{3}\alpha\left(h-\frac{1}{2}\right)n+2\left(h+\frac{1}{2}\right)\sqrt{hn}\\
	&= \frac{(5h-2)\alpha n}{6},
	\end{split}
\end{equation*}
where the last second inequality follows from $n\geq \frac{t^2}{\alpha^3}\geq \frac{(12h+6)^2h}{(3h-1)^2\alpha ^2}$.

Suppose to the contrary that there is a vertex $v\in R'$ with $d_{G}(v)\leq \frac{1}{3}\alpha n$. By the fact  $$\sum_{u\in (N^2(v)\cap R)}d_{G[N(v)]}(u)= e(N(v),\,N^2(v)\cap R)\leq e(G[(N(v)\cup R)\setminus\{v\}]),$$  and the above claim, we obtain 
\begin{equation}\label{2.5}
	\sum_{u\in (N^2(v)\cap R)}d_{G[N(v)]}(u)x_u\leq e(N(v),\,N^2(v)\cap R)x_z\leq \frac{(5h-2)\alpha nx_z}{6}.
\end{equation} 

 From
\begin{equation*}
	4\sqrt{hn}\alpha x_z\leq \rho x_v= \sum_{u\in N(v)}x_u\leq d_{G}(v)x_z,
\end{equation*}
 we have $d_{G}(v)\geq 4\sqrt{hn}\alpha \geq t$. Since $G\big[N\left[v\right]\big]$ is $kS_{\ell -1}\cup P_{\ell}$-free, we obtain 
  $e\left(G\big[N\left[v\right]\big]\right)\leq (h+\frac{1}{2})d_{G}(v)$.  Therefore, $e(G[N(v)])\leq \left(h-\frac{1}{2}\right)d_{G}(v)$ holds. It follows
\begin{equation}\label{2.4}
	\left(d_{G}(v)+2e(G[N(v)])\right)x_z\leq 2hd_{G}(v)x_z\leq \frac{2}{3}\alpha hnx_z.
\end{equation}

Moreover, from Lemma \ref{u2}, we have $e(N(v),\,N^2(v)\setminus R)\leq hn $,
which implies  
\begin{equation}\label{2.6}
	\sum_{u\in (N^2(v)\setminus R)}d_{G[N(v)]}(u)x_u\leq e(N(v),\,N^2(v)\setminus R)\alpha x_z\leq  \alpha h n x_z.
\end{equation}

Combining (\ref{2.5}) \textendash $\,$ (\ref{2.6}), we have
\begin{align}
	\rho ^2x_v&=d_{G}(v)x_v+\sum_{u\in N(v)}d_{G[N(v)]}(u)x_u +\sum_{u\in N^2(v)}d_{G[N(v)]}(u)x_u  \notag \\
	  	&\leq \big(d_{G}(v) +2e(G[N(v)])\big)x_z +\sum_{u\in (N^2(v)\cap R)}d_{G[N(v)]}(u)x_u+\sum_{u\in (N^2(v)\setminus R)}d_{G[N(v)]}(u)x_u\notag \\
        &\leq \frac{2}{3}\alpha hnx_z+\frac{(5h-2)\alpha nx_z}{6} +\alpha  hn x_z\notag.
\end{align}

On the other hand, by  (\ref{3-1}) we have 
\begin{equation*}
	\rho ^2x_v\geq hnx_v\geq 4\alpha h nx_z.
\end{equation*}

Then we obtain the following inequality
\begin{equation*}
	4\alpha hnx_z\leq \frac{2}{3}\alpha h nx_z+\frac{(5h-2)\alpha nx_z}{6}+ \alpha hn x_z.
\end{equation*}
Simplifying the inequality, we obtain $h\leq -\frac{2}{9}$, which is a contradiction. Therefore, each vertex in $R'$ has degree larger than $\frac{1}{3}\alpha n$.
\end{proof}
\begin{lemma}\label{3-3-2}
	$\lvert R'\rvert\leq \frac{3(h+1)}{\alpha}$.
	\end{lemma}
\begin{proof}
If $\lvert R'\rvert \leq t$, then $\lvert R'\rvert\leq \frac{3(h+1)}{\alpha}$ holds as $\alpha =\frac{1}{2(h+1)t^2}$. Now we suppose $\lvert R'\rvert >t$, then by (\ref{upbound}) we have $e(G[R'])\leq \left(h+\frac{1}{2}\right)\lvert R'\rvert$. By Lemma \ref{3-3}, $\sum_{v\in R'}d_{G}(v)\geq \frac{1}{3}\alpha n\lvert R'\rvert $ holds, and we have
\begin{equation*}
	\sum_{v\in \overline{R'}}d_{G}(v)\geq e(\overline{R'},\, R')=\sum_{v\in R'}d_{G}(v)-2e(G[R'])\geq \frac{1}{3}\alpha n\lvert R'\rvert -2\left(h+\frac{1}{2}\right)\lvert R'\rvert.
\end{equation*}
Therefore we have 
\begin{equation*}
	\begin{split}
	\left(h+\frac{1}{2}\right)n&\geq e(G)=\frac{1}{2}\sum_{v\in R'}d_{G}(v)+\frac{1}{2}\sum_{v\in \overline{R'}}d_{G}(v)\\
	&\geq \frac{1}{6}\alpha n \lvert R'\rvert +\frac{1}{6}\alpha n \lvert R'\rvert-\left(h+\frac{1}{2}\right)\lvert R'\rvert\\
	&=\frac{1}{3}\alpha n \lvert R'\rvert -\left(h+\frac{1}{2}\right)\lvert R'\rvert.
	\end{split}
\end{equation*}
Noting $n\geq \frac{t^2}{\alpha ^3}$, we have $\alpha n\geq (6h+3)(h+1)$, together with $\left(h+\frac{1}{2}\right)n\geq \frac{1}{3}\alpha n \lvert R'\rvert -\left(h+\frac{1}{2}\right)\lvert R'\rvert$, we obtain
\begin{equation*}
	\lvert R'\rvert\leq \frac{(6h+3)n}{2\alpha n-6h-3}\leq \frac{3(h+1)}{\alpha}.
	\end{equation*}
\end{proof}

\begin{lemma}\label{3-4}
	If $v $ is a vertex  with $x_v=mx_z$ and $\frac{1}{2(h+1)}\leq m\leq 1$, then $d_{G}(v)>\left(m-\frac{1}{6(h+1)}\right)n$.
\end{lemma}	
	\begin{proof}
	Suppose to the contrary that there is a vertex $v$ with $x_v=mx_z$ $\left(\frac{1}{2(h+1)}\leq m\leq 1\right)$ and $d_{G}(v)\leq \left(m-\frac{1}{6(h+1)}\right)n.$ By the definition of $R'$, $v\in R'$ holds.
	Let $M=N(v)\cup N^2(v)$. From (\ref{upbound}) and (\ref{,2}), we obtain
\begin{equation}\label{2.9}
	e(M\setminus R',\, N(v))\leq 2e(G)\leq 2\left(h+\frac{1}{2}\right)n.
\end{equation}

Since $v\in R'$, we have $d_G(v)>\frac{1}{3}\alpha n$  from Lemma \ref{3-3}. Then $G[R'\setminus \{v\},\,N(v)\setminus R']$ is $(k-1)S_{\ell -1}\cup P_{\ell}$-free. Otherwise we may add a copy of $S_{\ell -1}$ with $v$  as center vertex to obtain a copy of  $kS_{\ell -1}\cup P_{\ell}$ in $G$. Then by Lemma \ref{u2} we have
$$e(R'\setminus \{v\},\,N(v)\setminus R')\leq (h-1)(d_G(v)+\lvert R'\rvert -1).$$
Then from Lemma \ref{u2} and (\ref{,1}) we have
\begin{align}\label{2.10}
	e\left(M\cap R',N(v)\right)&\leq e(R'\setminus \{v\},\, N(v)) \notag \\
	&= e(R'\setminus \{v\},\,N(v)\setminus R')+2e\left(G[R'\cap N(v)]\right)+e\left(R'\setminus N[v],\, R'\cap N(v)\right) \notag \\
	&\leq (h-1)(d_G(v)+\lvert R'\rvert -1)+2e(G[R']).
\end{align}

If $\lvert R'\rvert \geq t$, then by (\ref{upbound}), $e(G[R'])\leq (h+\frac{1}{2})\lvert R'\rvert$ holds. From Lemma \ref{3-3-2}, $\lvert R'\rvert \leq \frac{3(h+1)}{\alpha}$ holds, then we have 
\begin{align}\label{2.11}
	(h-1)\lvert R'\rvert +2e(G[R'])&\leq (h-1)\lvert R'\rvert +(2h+1)\lvert R'\rvert \leq \frac{9h(h+1)}{\alpha}.
\end{align}

If $\lvert R'\rvert < t$, then $e(G[R'])<\frac{t(t-1)}{2}$ holds. We have
\begin{align}\label{2.12}
	(h-1)\lvert R'\rvert +2e(G[R'])\leq (h-1)t+t(t-1)\leq \frac{9h(h+1)}{\alpha},
\end{align}
the last inequality holds as $\alpha =\frac{1}{2(h+1)t^2}$.

For each vertex $u\in M\setminus R'$, the definition of $R'$ implies $x_u\leq 4\alpha x_z$. Therefore, combining (\ref{2.9}) \textendash $\;$(\ref{2.12}), we have
\begin{equation*}\label{2.8}
	\begin{split}
	\rho ^2x_v&=d_{G}(v)x_v+\sum_{u\in (M\setminus R')}d_{G[N(v)]}(u)x_u+\sum_{u\in (M\cap R')}d_{G[N(v)]}(u)x_u \notag\\
	&\leq d_{G}(v)x_v+e(M\setminus R', \,N(v))4\alpha x_z +e(M\cap R', \,N(v))x_z\\
	&\leq d_{G}(v)mx_z+8(h+\frac{1}{2})\alpha n x_z+(h-1)d_{G}(v)x_z+\frac{9h(h+1)}{\alpha}x_z-(h-1)x_z.
	\end{split}
\end{equation*}

On the other hand, by  (\ref{3-1}), we have 
\begin{equation*}
	\rho ^2x_v\geq hn x_v=hnmx_z.
\end{equation*}
Therefore, we obtain
\begin{equation*}
hnm\leq d_{G}(v)m+ 8\left(h+\frac{1}{2}\right)\alpha n+(h-1)d_{G}(v)+\frac{9h(h+1)}{\alpha}-(h-1).
\end{equation*}
Since $d_{G}(v)\leq \left(m-\frac{1}{6(h+1)}\right)n$,  we have
\begin{equation*}
	-nm^2+\frac{6h+7}{6(h+1)}nm+\frac{h-1}{6(h+1)}n-8\alpha \left(h+\frac{1}{2}\right)n\leq \frac{9h(h+1)}{\alpha}-(h-1).
\end{equation*}
Let $$g(x)=-nx^2+\frac{6h+7}{6(h+1)}nx+\frac{h-1}{6(h+1)}n-8\alpha (h+\frac{1}{2})n.$$ Then $g(x)\geq g(1)$ holds when $x\in [\frac{1}{2(h+1)},\,1]$. So $g(m)\geq g(1)$ holds and we have
\begin{equation*}
	\left(\frac{h}{6(h+1)}-8\alpha (h+\frac{1}{2})\right)n\leq \frac{9h(h+1)}{\alpha}-(h-1).
\end{equation*}
Then we get $n\leq 54h(h+1)^4t^2$, while it contradicts the fact $n\geq 8(h+1)^3t^8$. 
\end{proof}

Let $R''=\{v\in V(G):\,x_v\geq \frac{1}{2(h+1)}x_z\}$. Clearly $R''\subseteq R'\subseteq R$ holds. To prove the cardinality of $R''$, we prove a lower bound of the degree of vertices in $R''$.
\begin{lemma}\label{3-5}
	For each $v$ in $R''$, we have $d_{G}(v)\geq \left(1-\frac{5}{6(h+1)}\right)n$.
\end{lemma}
\begin{proof}
	By Lemma \ref{3-4}, it suffices  to prove $x_v\geq \left(1-\frac{2}{3(h+1)}\right)x_z$ for every $v\in R''$. Suppose to the contrary that there is a vertex $v\in R''$ such that $x_v=mx_z$, where $\frac{1}{2(h+1)}\leq m<1-\frac{2}{3(h+1)}$. Let $R_1=N(z)\cap R'$, $R_2=N^2(z)\cap R'$, $S_1=N(z)\setminus R_1$ and $S_2=N^2(z)\setminus R_2$. 
	By (\ref{upbound}) and Lemma \ref{u2}, we have 
	\begin{equation*}
		\begin{split}
			2e(S_1)\leq 2\left(h+\frac{1}{2}\right)n,\quad 
			e(R_1,\,S_1)\leq hn\quad{\rm{and}}\quad
			e(N(z),\,S_2)\leq hn.
		\end{split}
	\end{equation*}
	From Lemma  \ref{3-3-2} we have $\lvert R'\rvert \leq \frac{3(h+1)}{\alpha}$, therefore we deduce 
	\begin{align}
		&2e(G[S_1])4\alpha +2e(G[R_1\cup R_2])+e(R_1,\,S_1)4\alpha +e(N(z),\,S_2)4\alpha \notag \\
		\leq& 2\left(h+\frac{1}{2}\right)n4\alpha +2\binom{\lvert R'\rvert}{2}+hn4\alpha +hn4\alpha \notag \\
		\leq& 19\alpha h n.
	\end{align}
Noting that $$e(S_1,(R_1\cup R_2\cup \{z\})\setminus\{v\})+\lvert N(z)\cap N(v)\rvert=e(S_1,\,R_1\cup R_2\cup \{z\})+\lvert R_1\cap N(v)\rvert$$ holds. By Lemma \ref{u2}, we have $e(S_1,\,R_1\cup R_2\cup \{z\})\leq hn.$ Then
	\begin{align*}
		\rho ^2x_z&=\sum_{u\sim z}\sum_{w\sim u}x_w\notag \\
		&\leq d_G(z)x_z+\sum_{u\in S_1}\sum_{w\in (R_1\cup R_2),\atop w\sim u}x_w+2e(G[S_1])4\alpha x_z\\
		&\quad\quad +2e(G[R_1\cup R_2])x_z+e(R_1,\,S_1)4\alpha x_z+e(N(z),\,S_2)4\alpha x_z\notag\\
		&\leq e(S_1,\,(R_1\cup R_2\cup \{z\})\setminus\{v\})x_z+\lvert N(z)\cap N(v)\rvert x_v+\lvert R'\rvert x_z+19\alpha h nx_z\notag \\
		&\leq e(S_1,\,(R_1\cup R_2\cup \{z\})\setminus\{v\})x_z+\lvert N(z)\cap N(v)\rvert \left(1-\frac{2}{3(h+1)}\right)x_z+20\alpha hnx_z\notag \\
		&=e(S_1,\,R_1\cup R_2\cup \{z\})x_z+\lvert R_1\cap N(v)\rvert x_z-\frac{2}{3(h+1)}\lvert N(z)\cap N(v)\rvert x_z+20\alpha hnx_z\\
		&\leq hnx_z+\lvert R'\rvert x_z-\frac{2}{3(h+1)}d_{G[N(z)]}(v)x_z+20\alpha hnx_z.
	\end{align*}
Together with the fact that $hnx_z\leq \rho ^2x_z$ and $\lvert R'\rvert \leq \frac{3(h+1)}{\alpha}$, we obtain
\begin{equation*}
	\frac{2}{3(h+1)}d_{G[N(z)]}(v)\leq 20\alpha hn+\frac{3(h+1)}{\alpha }.
\end{equation*}

Since $ \alpha =\frac{1}{2(h+1)t^2}$ and $n\geq \frac{t^2}{\alpha ^3}$, we get 
\begin{align}\label{m}
		d_{G[N(z)]}(v)\leq 30(h+1)\alpha hn+\frac{9(h+1)^2}{2\alpha}<\frac{n}{6(h+1)}.
\end{align}
	Since $z,\,v\in R'$, by Lemma \ref{3-4} we have $d_G(v)>\left(m-\frac{1}{6(h+1)}\right)n>\frac{n}{3(h+1)}$ and $d_G(z)>\left(1-\frac{1}{6(h+1)}\right)n$. So 
	$$d_{G[N(z)]}(v)\geq \frac{n}{3(h+1)}-\frac{n}{6(h+1)}=\frac{n}{6(h+1)}$$ holds, a contradiction to (\ref{m}). Hence we have $d_G(v)\geq (1-\frac{5}{6(h+1)})n$ for each $v$ in $R''$.
\end{proof}

 Now we prove the exact value of $\lvert R''\rvert$.
 \begin{lemma}\label{3-6}
 	 $\lvert R''\rvert=h$. 
 \end{lemma}
\begin{proof}
	Firstly, we prove $\lvert R''\rvert>h-1$. Suppose to the contrary that $\lvert R''\rvert\leq h-1$. Let $M=N(z)\cup N^2(z)$. By (\ref{upbound}) and (\ref{,2}), we obtain 
	\begin{align}\label{2.16}
		e(M\setminus R'',\,N(z))\leq 2e(G)\leq 2\left(h+\frac{1}{2}\right)n.
	\end{align}
	The fact $M\cap R'' \subseteq R''\setminus\{z\}$ and (\ref{,3}) imply
	\begin{equation}\label{2.17}
		e(M\cap R'',\,N(z))\leq \lvert M\cap R''\rvert\cdot\lvert N(z)\rvert\leq \lvert R''\setminus \{z\}\rvert\, d_G(z).
	\end{equation}

	Combining (\ref{2.16}) and (\ref{2.17}), we have
	\begin{align}\label{2.15}
		\rho ^2x_z&=d_G(z)x_z+\sum_{u\in (M\cap R'')}d_{G[N(z)]}(u)x_u+\sum_{u\in (M\setminus R'')}d_{G[N(z)]}(u)x_u \notag \notag \\
		&\leq d_G(z)x_z+e(M\cap R'',\,N(z))x_z+e(M\setminus R'',\,N(z))\frac{1}{2(h+1)}x_z\notag \\
		&\leq d_G(z)x_z+\lvert R''\setminus \{z\}\rvert d_G(z) x_z+2\left(h+\frac{1}{2}\right)n\frac{1}{2(h+1)}x_z \notag \\
		&=\lvert R''\rvert d_G(z)x_z +\frac{2h+1}{2h+2}nx_z\notag .
	\end{align}

Together with the fact $hn\leq \rho ^2(G)$, it follows
\begin{equation*}
	hn\leq \lvert R''\rvert d_G(z)+\frac{2h+1}{2h+2}n\leq (h-1)(n-1)+\frac{2h+1}{2h+2}n,
\end{equation*}
which may deduce $h\leq 0$, a contradiction. Hence, $\lvert R''\rvert >h-1$.

If $\lvert R''\rvert \geq h+1$, suppose $\{v_1,\,v_2,\cdots,v_{h+1}\}\subseteq R''$ and $W=N(v_1)\cap \cdots \cap N(v_{h+1})$. By (\ref{cup}), it follows
\begin{equation*}
	\begin{split}
	\lvert W\rvert  &\geq \sum_{i=1}^{h+1}d_G(v_i)-hn \geq (h+1)\left(1-\frac{5}{6(h+1)}\right)n-hn\geq (k+1)\ell,
	\end{split}
\end{equation*}
the last inequality holds as $n\geq \frac{t^2}{\alpha^3}$.
When $\ell $ is even, we may find a path 
$P_{\ell}=u_1v_1\cdots u_{\frac{\ell}{2}}v_{\frac{\ell}{2}}$, where $u_i\in W$ ($i= 1,\,2,\,\cdots,\frac{\ell}{2}$). As $d_G(v_i)\geq \left(1-\frac{5}{6(h+1)}\right)n$ ($i=\frac{\ell}{2}+1,\cdots,h+1$), we may find a copy of $kS_{\ell-1}$ in $G-P_{\ell}$ with $v_{\frac{\ell}{2}+1},\cdots,v_{h+1}$ as the $k$ center vertices.
Then there is a copy of $kS_{\ell -1}\cup P_{\ell}$ in $G$, a contradiction. When $\ell $ is odd, we may find a path
$P_{\ell}=u_1v_1\cdots u_{\frac{\ell-1}{2}}v_{\frac{\ell-1}{2}}u_{\frac{\ell+1}{2}}$ where  $u_i\in W$ ($i= 1,\,2,\,\cdots,\frac{\ell+1}{2}$) and a copy of   $kS_{\ell-1}$ in $G-P_{\ell}$ with $v_{\frac{\ell+1}{2}},\cdots,v_{h+1}$ as the $k$ center vertices.
Then there is a copy of $kS_{\ell -1}\cup P_{\ell}$ in $G$, a contradiction.
Therefore, $\lvert R''\rvert =h$ holds.
\end{proof}

\section{Proof of Theorem \ref{2}}
Suppose $k\geq 1$, $\ell \geq 4$,  $h=k+\lfloor\frac{\ell }{2}\rfloor-1$, $t=(\ell ^2-\ell +1)k+\frac{\ell ^2+3\ell -2}{2}$, $\alpha =\frac{1}{2(h+1)t^2}$ and $n\geq \frac{t^2}{\alpha ^3}$.
Let $G\in {\rm Ex}_{sp}(n,\,kS_{\ell -1}\cup P_{\ell})$, $x_z={\rm max}\{x_v:v\in V(G)\}$ and $R''=\{v\in V(G):\,x_v\geq \frac{1}{2(h+1)}x_z\}$. By Lemma \ref{3-6}, $\lvert R''\rvert =h$ holds. Write $R''=\{v_1,\,v_2,\cdots,v_{h-1},\,z\}$ and $W=\{v\in V(G): R''\subseteq N_G(v)\}$. By Lemma \ref{3-5}, we have $d_G(v)\geq n-\frac{5n}{6(h+1)}$ for each $v\in R''$. Hence when $n\geq 8(h+1)^3t^8$, we have $$\lvert W\rvert \geq n-\frac{5n\lvert R''\rvert}{6(h+1)} >(k+1)\ell. $$ 

\textbf{Claim A.}
$G[\overline{R''}]$ is $S_{\ell -1}$-free.

	 Suppose to the contrary that  $G[\overline{R''}]$ contains a copy of $S_{\ell -1}$. Let $R_1=\{v_1,\,v_2,\cdots,v_{k-1}\}\subseteq R''$. Since  each vertex in $R''$ has degree at least $\left(1-\frac{5}{6(h+1)}\right)n$, we may find a $(k-1)S_{\ell -1}$ with $v_1,\cdots,v_{k-1}$ as the $k-1$ center vertices and leaves vertices in $V(G)\setminus (R''\cup V(S_{\ell -1}))$. Then there is a copy of $kS_{\ell -1}$ in $G$.
	When $\ell $ is even, we have $\lvert R''\setminus R_1\rvert =\frac{\ell}{2}$. Since $\lvert W\rvert -k\ell>\ell$ when $n\geq \frac{t^2}{\alpha ^3}$, we may find a copy of $P_{\ell}=v_{k}u_1\cdots zu_{\frac{\ell}{2}}$, where $u_i\in W\setminus V\left(kS_{\ell -1}\right)$ ($i=1,\cdots ,\frac{\ell}{2}$). Then $G$ contains a copy of $kS_{\ell -1}\cup P_{\ell}$, a contradiction. When $\ell$  is odd,  we have $\lvert R''\setminus R_1\rvert =\frac{\ell-1}{2}$. We may find a copy of $P_{\ell}=u_1v_k\cdots u_{\frac{\ell-1}{2}}zu_{\frac{\ell+1}{2}}$, where $u_i\in W\setminus V\left(kS_{\ell -1}\right)$ ($i=1,\,\cdots ,\,\frac{\ell+1}{2}$), a contradiction. Therefore, $G[\overline{R''}]$ is $S_{\ell -1}$-free.

\textbf{Claim B.}
	$z$ is adjacent to each vertex of $\overline{R''}$.

 Suppose to the contrary that there is a vertex $v$ in $\overline{R''}$ not adjacent to $z$. Let $G_1$ be a graph obtained from $G$ by deleting all edges incident with $v$ in $G[\overline{R''}]$ and adding an edge $zv$. If $G_1$ contains a copy of $kS_{\ell -1}\cup P_{\ell }$, we have $zv\in E(kS_{\ell -1}\cup P_{\ell })$. Since $zv$ is a pendent edge in $G_1$ and $\lvert W\rvert >(k+1)\ell$, we may find a vertex $w\in W$ and use $zw$ as a replacement of $zv$, then we obtain a copy of $kS_{\ell -1}\cup P_{\ell }$ in $G$, a contradiction. Therefore, $G_1$ is $kS_{\ell -1}\cup P_{\ell }$-free. 
 The fact that $G[\overline{R''}]$ is $S_{\ell -1}$-free implies $d_{G[\overline{R''}]}(v) < \ell -1$. When $w\notin R''$, we have $x_w<\frac{1}{2(h+1)}x_z$.  It follows 
$$
		\sum_{w\in N_{G[\overline{R''}]}(v)}x_w<\frac{(\ell -1)x_z}{2(h+1)}.
$$
Then we have
\begin{equation*}
	\begin{split}
		\rho (G_1)-\rho (G)&\geq 2x_v\left(x_z-\sum_{w\in N_{G[\overline{R''}]}(v)}x_w\right)>2x_v\left(x_z-\frac{(\ell -1)x_z}{2(h+1)}\right)>0,
	\end{split}
\end{equation*}
the inequality deduces that $z$ is adjacent to each vertex of $\overline{R''}$.

Set $R_2=\{v_1,\cdots,v_k\}\subseteq R''$.

\textbf{Claim C.}
	When $\ell $ is even, there is no edge in $G[\overline{R''}]$.

	Suppose to the contrary that there is an edge in $G[\overline{R''}]$, denoted by $uv$. By Claim B, we have $uz$, $vz\in E(G)$.  Since each vertex in $R''$ has degree at least $\left(1-\frac{5}{6(h+1)}\right)n$, we may find a copy of $kS_{\ell -1}$ in $G$ with  the $k$ center vertices in $R_2$ and leaves vertices in $\overline{R''}$. Then $\lvert R''\setminus R_2\rvert =\frac{\ell }{2}-1$ holds. As $\lvert W\rvert >(k+1)\ell$, we may find a copy of $P_{\ell}=u_1v_{k+1}\cdots u_{\frac{\ell}{2}-2} v_{h-1}u_{\frac{\ell}{2}-1}zvu$, where $u_i\in W\setminus V(kS_{\ell -1})$ ($i=1,\cdots, \frac{\ell}{2}-1$). Then $G$ contains a copy of $kS_{\ell -1}\cup P_{\ell}$, a contradiction.


\textbf{Claim D.}
	When $\ell$ is odd,  there is at most  one edge in $G[\overline{R''}]$.

	Suppose to the contrary that there is a $P_3\subseteq G[\overline{R''}]$, denoted by $uvw$. By Claim B, we have $wz\in E(G)$. Since each vertex in $R''$ has degree at least $\left(1-\frac{5}{6(h+1)}\right)n$, we may find a copy of $kS_{\ell -1}$ in $G$ with the $k$ center vertices in $R_2$ and leaves vertices in $\overline{R''}$. Then $\lvert R''\setminus R_2\rvert =\frac{\ell-3 }{2}$ holds. As $\lvert W\rvert >(k+1)\ell$, we may find a path $P_{\ell}=u_1v_{k+1}u_2\cdots v_{h-1} u_{\frac{\ell-3}{2}}zwvu$ where   $u_i\in W\setminus V(kS_{\ell -1})$ ($i=1,\cdots ,\frac{\ell-3}{2}$). Then there is a copy of $kS_{\ell -1}\cup P_{\ell }$ in $G$, a contradiction. Hence $G$ is $P_3$-free.

	Suppose to the contrary $2P_2\subseteq G[\overline{R''}]$, denoted by $w_1w_2$, $w_3w_4$. If one of $\{w_1,\,w_2,\,w_3,\,w_4\}$ is adjacent to a vertex  in $ R''\setminus \{z\}$, without loss of generality, say $v_{\frac{\ell-5}{2}}w_1\in E(G)$. We may find a $P_{\ell}$ with $\frac{\ell-3}{2}$ vertices in $R''$, $P_{\ell}=u_1v_1\cdots u_{\frac{\ell-5}{2}}v_{\frac{\ell-5}{2}}w_1w_2zw_3w_4$, $u_i\in W$ ($i=1,\,\cdots ,\,\frac{\ell-5}{2}$). Moreover, we may find a $kS_{\ell -1}$ with the remaining $k$ vertices of $R''$ as the $k$ center vertices and leaves vertices in $\overline{R''}\setminus P_{\ell}$, a contradiction.
 Hence we have $N_{G[R'']}(w_i)=\{z\}$ ($i=1,\,2,\,3,\,4$). Let $G_2$ be a graph obtained from $G$ by deleting edge $w_1w_2$ and adding edge $w_1v_1$. If $G_2$ contains a copy of  $kS_{\ell -1}\cup P_{\ell }$, we have $w_1v_1\in E(kS_{\ell -1}\cup P_{\ell })$. Since $w_1v_1$ is a pendent edge in $G_2$ and $\lvert W\rvert>(k+1)\ell$, we may find a vertex $w_5\in W$ and use $w_5v_1$ as a replacement of $w_1v_1$. Then we obtain  a copy of $kS_{\ell -1}\cup P_{\ell }$ in $G$, a contradiction.  Hence
	$G_2$ is $kS_{\ell -1}\cup P_{\ell }$-free. Since $v_1\in R''$, $w_2\in \overline{R''}$, we have 
$
		x_{w_2}<\frac{1}{2(h+1)}x_z\leq x_{v_1}.
$
Then
\begin{equation*}
	\rho (G_2)-\rho(G)\geq \textbf{x}^T(A(G_2)-A(G))\textbf{x}=2x_{w_1}x_{v_1}-2x_{w_1}x_{w_2}>0,
\end{equation*}
this contradiction implies that $G[\overline{R''}]$ is 2$P_2$-free. Hence there is at most one edge in $G[\overline{R''}]$.

When $\ell $ is even,  $G[\overline{R''}]=\overline{K}_{n-h}$ holds from Claim C. By the facts that the spectral radius of a graph does not decrease  by adding an edge and $S_{n,\,h}$ is $kS_{\ell -1}\cup P_{\ell}$-free, each vertex in $R''$ is adjacent to each vertex in $\overline{R''}$ and $G[R'']=\{K_n\}$ holds. Then we have $G=S_{n,\,h}$. 
When $\ell$ is odd,  $G[\overline{R''}]=K_2\cup \overline{K}_{n-h-2}$ holds from Claim D. Since $S^{+}_{n,\,h}$ is $kS_{\ell -1}\cup P_{\ell }$-free, each vertex in $R''$ is adjacent to each vertex in $\overline{R''}$ and $G[R'']=\{K_n\}$ holds.  Then we have $G=S^{+}_{n,\,h}$.

	~\\
	{\textbf{Declaration}}
	
	The authors have declared that no competing interest exists.
	

\begin{thebibliography}{99}
		\small{	\bibitem{2008}A. Bhattacharya, S. Friedland, U. N. Peled, On the first eigenvalue of bipartite graphs, Electron. J. Combin. 15 (2008) Article R144 23pp.
			\bibitem{2016}H. Bielak, S. Kieliszek, The Tur\'{a}n number of the graph $2P_5$, Disuss. Math. Graph Theory. 36 (2016) 683-694.
			\bibitem{2011}N. Bushaw, N. Kttle, Tur\'{a}n numbers of multiple paths ad equibipartite forests, Combin. Probab. Comput. 20 (6) (2011) 837-853.
			\bibitem{2008-P} P. Balister, E. Gy\H{o}ri, J. Lehel, R. Schelp, Connected graphs without long paths, Discrete Math. 308 (2008) 4487-4494.
			\bibitem{2019}M. Chen, A. Liu, X. Zhang, Spectral extremal results with forbidding linear forests, Graphs Combin. 35 (2019) 335-351.
			\bibitem{2021-M}M. Chen, A. Liu, X. Zhang, On the spectral radius of graphs without a star forest, Discrete Math. 344 (2021) 112269.
			\bibitem{2020}S. Cioab\v{a}, L. Feng, M. Tait, X. Zhang, The maximum spectral radius of graphs without friendship subgraphs, Electron. J. Combin. 27 (2020) P4.22.
			\bibitem{2022-S}S. Cioab\v{a}, D. Desai, M. Tait, The spectral radius of graphs with no odd wheels, European J. Combin. 99 (2022), Paper No. 103420, 19pp.
			\bibitem{1959} P. Erd\H{o}s, T. Gallai, On maximal paths and circuits of graphs, Acta Math. Acad. Sci. Hungar 10 (1959) 337-356.
			\bibitem{2022-F}T. Fang, X. Yuan, Some results on the Tur\'{a}n number of $k_1P_{\ell}\cup k_2S_{\ell -1}$, arxiv:221109432.
			\bibitem{2007}L. Feng, G. Yu, X. Zhang, Spectral radius of graphs with given matching number, Linear Algebra Appl. 422 (2007) 133-138.
			\bibitem{1998}B. Guiduli, Spectral Extremal for Graphs (Ph.D. thesis), University of Chicago, 1998.
			\bibitem{2019-G}J. Gao, X. Hou, The spectral radius of graphs without long cycles, Linear Algebra Appl. 566 (2019) 17-33.
			\bibitem{2001}Y. Hong, J. Shu, K. Fang, A sharp upper bound of the spectral radius of graphs, J. Combin. Theory Ser. B 81 (2001) 177-183.		
			\bibitem{1997}G. Kopylov, Maximal paths and cycles in a graph, Dokl. Akad. Nauk SSSR 23 (1997) 19-21.
			\bibitem{2013}B. Lidick\'{y}, H. Liu, C. Palmer, On the Tur\'{a}n number of forests, Electron. J. Combin. 20 (2) (2013) 62.
			\bibitem{2022-L}S. Li, J. Yin, J. Li, The Tur\'{a}n number of $kS_{\ell}$, Discrete Math. 345 (1) (2022) 112653.
			\bibitem{2019-L}Y. Lan, T. Li, Y. Shi, J. Tu, The Tur\'{a}n number of star forests, Appl. Math. Comput. 348 (2019) 270-274.
			\bibitem{2002}V. Nikiforov, Some inequalities for the largest eigenvalue of a graph, Combin. Probab. Comput. 11 (2002) 179-189.
			\bibitem{2010}V. Nikiforov, The spectral radius of graphs without paths and cycles of specified length, Linear Algebra Appl. 432 (2010) 2243-2256.
			\bibitem{2017-N}V. Nikiforov, Bounds on graph eigenvalues II, Linear Algebra Appl. 427 (2017) 183-189.
			\bibitem{2017}M. Tait, J. Tobin, There conjectutres in extremal spectral graph theory, J. Combin. Theory Ser. B 126 (2017) 137-161.
			\bibitem{2023}J. Wang, L. Kang, Y. Xue, On a conjecture of spectral extremal problems, J. Combin. Theory Ser. B 159 (2023) 20-41.
			\bibitem{2021}L. Yuan, X. Zhang, Tur\'{a}n numbers for disjoint paths, J. Graph. Theory. 98 (3) (2021) 499-524.
			\bibitem{2017}L. Yuan, X. Zhang, The Tur\'{a}n number of disjoint copies of paths, Discrete Math. 340 (2) (2017) 132-139.  	
	    	\bibitem{2012}M. Zhai, B. Wang, Proof of a conjecture on the spectral radius of $C_4$-free graphs, Linear Algebra Appl. 437 (2012) 1641-1647.
	    	\bibitem{arxiv2}M. Zhai, R. Liu, A spectral Erd\H{o}s-P\'{o}sa Theorem, Arxiv: 2208.02988v1.
	    
	    	
}
	\end{thebibliography}
\end{document}